\documentclass[12pt]{amsart}

\usepackage{amsmath,amssymb,mathtools}
\usepackage{enumitem}
\usepackage[nocompress]{cite}
\usepackage{geometry}
\usepackage{hyperref}

\newtheorem{theorem}{Theorem}[section]
\newtheorem{proposition}[theorem]{Proposition}
\newtheorem{lemma}[theorem]{Lemma}
\newtheorem{problem}[theorem]{Problem}
\theoremstyle{definition}
\newtheorem{definition}[theorem]{Definition}
\newtheorem{remark}[theorem]{Remark}
\newtheorem{example}[theorem]{Example}

\numberwithin{equation}{section}
\newcommand{\Sc}{\mathrm{Sc}}
\newcommand{\R}{\mathbb{R}}
\DeclareMathOperator{\Hess}{Hess}
\DeclareMathOperator{\Ric}{Ric}
\DeclareMathOperator{\ind}{Ind}
\DeclareMathOperator{\Int}{Int}

\title[Positive scalar curvature on products of open manifolds]
{Positive Scalar Curvature on Products of Noncompact Manifolds}

\author{Lizhi Chen}
\address[Lizhi Chen]{School of Mathematics and Statistics, Lanzhou University}
\email{lizhi.chen.math@gmail.com}

\author{Kuntao Jin}
\address[Kuntao Jin]{Department of Mathematical Sciences, Tsinghua University}
\email{jkt25@mails.tsinghua.edu.cn}

\author{Milan Jovanovic}
\address[Milan Jovanovic]{Texas A\&M University}
\email{milankj@tamu.edu}

\begin{document}

\begin{abstract}
We prove that the product of three connected noncompact smooth manifolds
without boundary admits a complete metric of uniformly positive scalar
curvature. The proof uses the additivity of Morse indices on products
and an orientation-free version of Das's open Morse--surgery construction.
We also give an example of two open manifolds whose product admits no
complete metric of nonnegative scalar curvature.
\end{abstract}

\maketitle

\section{Introduction}

Gromov \cite[Section~4.7]{GromovFourLectures} posed two questions
about complete metrics on products of noncompact manifolds. The first asks for
pointwise positive scalar curvature on a product of two factors; the
second asks for a uniform positive lower bound on a product of three.
\begin{problem}\label{prob:products}
\begin{enumerate}[label=\textup{(\arabic*)},leftmargin=*,itemsep=.35em,topsep=.35em]
\item Let \(X_1\) and \(X_2\) be connected noncompact manifolds. Does
\(X_1\times X_2\) admit a complete Riemannian metric \(g\) with
\(\Sc_g>0\) everywhere? Here, no uniform positive lower bound is required.

\item Let \(X_1,X_2,X_3\) be connected noncompact manifolds. Does
\(X_1\times X_2\times X_3\) admit a complete Riemannian metric \(g\)
with \(\inf\Sc_g>0\)?
\end{enumerate}
\end{problem}
These questions are motivated by the following result.
\begin{theorem}[\cite{GL1983}]
    \begin{enumerate}[label=\textup{(\arabic*)},leftmargin=*,itemsep=.35em,topsep=.35em]
        \item $\mathbb T^n \times \mathbb R$ does not admit a complete metric with positive scalar curvature.
        \item $\mathbb T^n \times \mathbb R^2$ admits a complete positive scalar
            curvature metric, but not a complete metric of uniformly positive scalar curvature.
        \item $\mathbb T^n \times \mathbb R^3$ admits a complete metric of uniformly
            positive scalar curvature.
    \end{enumerate}
\end{theorem}

In this paper, we prove the three-factor statement and give a counterexample to the
two-factor statement, even with nonnegative scalar curvature in place
of positive scalar curvature.

\begin{theorem}\label{thm:three-factors}
Let \(X_i^{n_i}\), \(i=1,2,3\), be connected noncompact smooth manifolds
without boundary. Then \(X_1\times X_2\times X_3\) admits a complete
Riemannian metric \(g\) such that
\[
  \inf_{X_1\times X_2\times X_3}\Sc_g>0.
\]
By rescaling, the metric may be chosen with \(\Sc_g\geq1\).
\end{theorem}

The proof begins with a proper Morse function without local maxima on
each factor. On the product, the sum of these functions has critical
indices at most \(N-3\), where \(N=n_1+n_2+n_3\). Its handle
attachments therefore induce surgeries of codimension at least three
on the regular level sets. The surgery-trace construction gives positive
scalar curvature on each compact cobordism between levels. Following
Das \cite{Das2026}, we rescale these traces and join them by sufficiently long collars
to obtain both a uniform scalar-curvature bound and completeness.

For two factors, the index estimate gives only \(N-2\), so this
construction does not apply. In Section~\ref{sec:two-factors}, we give
a product of two open manifolds that admits no complete metric of
nonnegative scalar curvature.

Unless otherwise stated, manifolds are smooth and without boundary.
We use \emph{open} to mean noncompact and without boundary.

\bigskip

\noindent \textbf{Acknowledgements:} We would like to thank Zhizhang Xie and Bo Zhu for
very helpful comments.

\bigskip

\noindent \textbf{AI statement:} GPT 5.6 Sol and GPT 6 Astra were used to find
references and polish the exposition. The authors take full responsibility for
the results presented in the paper.

\section{Preliminaries}

\begin{definition}\label{def:Morse-exhaustion}
A smooth function \(f\colon V\to[0,\infty)\) is \emph{proper} if
inverse images of compact sets are compact, or equivalently if every
sublevel set \(f^{-1}([0,a])\) is compact.

A point \(p\in V\) is \emph{critical} if \(df_p=0\). At a critical
point, the Hessian \(\Hess_p f\) is a well-defined symmetric bilinear
form on \(T_pV\). The point is \emph{nondegenerate} if this form is
nondegenerate, and \(f\) is a \emph{Morse function} if all its critical
points are nondegenerate. A proper Morse function bounded below is
called a \emph{Morse exhaustion}. The \emph{Morse index}
\(\ind_f(p)\) is the number of negative eigenvalues of \(\Hess_p f\),
counted with multiplicity, or equivalently the maximal dimension of a
subspace on which \(\Hess_p f\) is negative definite.

A number \(a\) is a \emph{regular value} of \(f\) if \(df_p\neq0\)
for every \(p\in f^{-1}(a)\); otherwise it is a \emph{critical value}.
\end{definition}

\begin{remark}\label{rem:Morse-standard-facts}
By the Morse lemma \cite[Lemma~2.2]{Mil63}, a critical point \(p\)
of index \(\lambda\) has local coordinates in which
\[
  f(x)=f(p)-x_1^2-\cdots-x_\lambda^2
              +x_{\lambda+1}^2+\cdots+x_n^2.
\]
In particular, nondegenerate critical points are isolated. If \(a\)
is a regular value of a proper function \(f\), the regular level set
theorem \cite[Corollary~5.14]{Lee13} shows that \(f^{-1}([0,a])\)
is a compact manifold with boundary \(f^{-1}(a)\).
\end{remark}

\begin{proposition}\label{prop:Morse-exhaustion-no-max}
Every connected noncompact smooth manifold \(V^n\) admits a proper
Morse function \(f\colon V\to[0,\infty)\) with no local maxima.
In particular, every critical point of \(f\) has index at most \(n-1\).
\end{proposition}

\begin{proof}
Fix a background Riemannian metric on \(V\) and a positive continuous
exhaustion \(\rho\colon V\to[0,\infty)\). By a result of Napier and
Ramachandran \cite[Corollary~2.1]{NapierRamachandran2004}, there is a
smooth Morse function \(f\) satisfying
\[
  f>\rho,
  \qquad
  \Delta f>\rho>0.
\]
Each sublevel set of \(f\) is a closed subset of the corresponding
compact sublevel set of \(\rho\), so \(f\) is proper. At a local
maximum, the Hessian would be negative semidefinite and hence
\(\Delta f\leq0\), a contradiction. Thus \(f\) has no local maxima
and, being Morse, has no critical points of index \(n\).
\end{proof}

\begin{proposition}\label{prop:product-index-bound}
For \(i=1,\ldots,r\), let \(X_i^{n_i}\) carry a proper Morse exhaustion
\(f_i\colon X_i\to[0,\infty)\) with no critical points of index
\(n_i\). Set \(N=\sum_{i=1}^r n_i\). Then
\[
  F\colon X_1\times\cdots\times X_r\longrightarrow[0,\infty),
  \qquad
  F(x_1,\ldots,x_r)=\sum_{i=1}^r f_i(x_i),
\]
is a proper Morse exhaustion. At every critical point
\(p=(p_1,\ldots,p_r)\),
\[
  \ind_F(p)=\sum_{i=1}^r\ind_{f_i}(p_i)\leq N-r.
\]
\end{proposition}

\begin{proof}
Since the functions \(f_i\) are nonnegative,
\[
  F^{-1}([0,a])\subseteq\prod_{i=1}^r f_i^{-1}([0,a]).
\]
The right-hand side is compact, and the left-hand side is closed.
Thus \(F\) is proper.

Under the splitting \(T_p(X_1\times\cdots\times X_r)
=\bigoplus_iT_{p_i}X_i\), we have
\[
  dF_p(v_1,\ldots,v_r)=\sum_{i=1}^r(df_i)_{p_i}(v_i).
\]
Consequently, \(p\) is critical for \(F\) exactly when each \(p_i\)
is critical for \(f_i\). At such a point,
\[
  \Hess_pF=\bigoplus_{i=1}^r\Hess_{p_i}f_i.
\]
This direct sum is nondegenerate, and its index is the sum of the
indices of its summands. Hence \(F\) is Morse and
\[
  \ind_F(p)=\sum_{i=1}^r\ind_{f_i}(p_i)
  \leq\sum_{i=1}^r(n_i-1)=N-r.
\]
\end{proof}

\section{Surgery traces and scaling collars}

The surgery theorem of Schoen--Yau \cite{SY1979} and
Gromov--Lawson \cite{GL1980} preserves positive scalar curvature
under surgeries of codimension at least three. To construct a metric
on an open manifold, we need metrics on the cobordisms between regular
levels, not just on the levels themselves. We use the following trace
version, in the form given by Walsh \cite[p.~46]{Walsh2011}.
Here a compact cobordism from \(Y_-\) to \(Y_+\) is a compact
manifold \(W\) with \(\partial W=Y_-\sqcup Y_+\).

\begin{theorem}[Gromov--Lawson cobordism theorem]
\label{thm:trace}
Let \(W^N\) be a compact cobordism from the closed manifold \(Y_-\)
to the closed manifold \(Y_+\), with \(N\geq3\). Suppose that
\(W\) admits a Morse function \(f\colon W\to[0,1]\) such that
\[
  f^{-1}(0)=Y_-,\qquad f^{-1}(1)=Y_+,
\]
and that all critical points lie in the interior and have index at
most \(N-3\). For every positive-scalar-curvature metric \(h_-\)
on \(Y_-\), there is a positive-scalar-curvature metric \(\widehat G\)
on \(W\) satisfying
\[
  \widehat G=h_-+dt^2\quad\text{near }Y_-,
  \qquad
  \widehat G=h_++dt^2\quad\text{near }Y_+,
\]
for some positive-scalar-curvature metric \(h_+\) on \(Y_+\).
\end{theorem}

We recall the handle-by-handle construction.

\begin{proof}[Proof sketch]
After perturbing \(f\) away from the boundary, assume that its
critical values are distinct. The corresponding handle decomposition
writes \(W\) as a finite succession of elementary cobordisms and
product regions. Start with \(h_-+dt^2\) on a collar of \(Y_-\).

An index-zero critical point introduces an \(N\)-ball. Give it a
torpedo metric, which has positive scalar curvature for \(N\geq3\)
and is equal near the boundary to
\[
  dt^2+\delta^2g_{\mathrm{rd}}^{N-1},
\]
where \(g_{\mathrm{rd}}^{N-1}\) is the unit round metric on
\(S^{N-1}\). Extend the metrics on the other components by products.

A handle of index \(1\leq\lambda\leq N-3\) is attached along
\(S^{\lambda-1}\times D^{N-\lambda}\) and induces a surgery of
codimension
\[
  (N-1)-(\lambda-1)=N-\lambda\geq3
\]
on the incoming level. The single-handle Gromov--Lawson trace
construction gives a positive-scalar-curvature
metric on this elementary cobordism, equal to the prescribed incoming
metric plus \(dt^2\) near the incoming boundary and product near the
outgoing boundary. Glue it to the preceding region along the matching
product collars and continue through the next product region.
Repeating this construction over all handles gives \(\widehat G\).
The incoming collar is unchanged, and the outgoing boundary metric
has positive scalar curvature because the metric is product there.
\end{proof}

Applying Theorem~\ref{thm:trace} successively gives positive scalar
curvature on the compact pieces of an exhaustion, but their curvature
lower bounds may tend to zero. We will rescale each trace before
attaching it. The following lemma supplies a collar between the old
boundary metric and its rescaled copy, with a uniform curvature bound
and an arbitrarily large length; compare Das \cite[Lemma~3.1]{Das2026}.

\begin{lemma}\label{lem:scaling-collar}
Let \((Y^m,h)\) be closed, with \(\Sc_h\geq2\sigma\) for some
\(\sigma>0\). Given \(0<\rho\leq1\) and \(\ell>0\), there are
\(L\geq\ell\) and a metric \(H\) on \(Y\times[0,L]\) such that
\(\Sc_H\geq\sigma\) and
\[
  H=h+ds^2\quad\text{near }s=0,
  \qquad
  H=\rho^2h+ds^2\quad\text{near }s=L.
\]
\end{lemma}

\begin{proof}
For \(\rho=1\), take \(L=\ell\) and \(H=h+ds^2\). Otherwise,
put \(a=-\log\rho>0\) and choose a smooth nondecreasing function
\(\theta\colon[0,1]\to[0,1]\) that is zero near \(0\) and one near
\(1\). For \(L\geq\ell\), set
\[
  u(s)=a\theta(s/L),\qquad f(s)=e^{-u(s)},\qquad
  H=ds^2+f(s)^2h.
\]
The warped-product scalar-curvature formula is
\[
  \Sc_H=f^{-2}\Sc_h
  -m(m-1)\left(\frac{f'}f\right)^2-2m\frac{f''}f.
\]
Substituting \(f'/f=-u'\) and \(f''/f=(u')^2-u''\) gives
\begin{equation}\label{eq:warped-scalar}
  \Sc_H=e^{2u}\Sc_h-m(m+1)(u')^2+2mu''.
\end{equation}
Let \(A_1=\lVert\theta'\rVert_\infty\) and
\(A_2=\lVert\theta''\rVert_\infty\). Then
\[
  \lVert u'\rVert_\infty\leq\frac{aA_1}{L},
  \qquad
  \lVert u''\rVert_\infty\leq\frac{aA_2}{L^2}.
\]
Choose
\[
  L\geq\max\left\{\ell,
  \sqrt{\frac{m(m+1)a^2A_1^2+2maA_2}{\sigma}}\right\}.
\]
Since \(u\geq0\) and \(\Sc_h\geq2\sigma\), equation
\eqref{eq:warped-scalar} yields
\[
  \Sc_H\geq2\sigma
  -m(m+1)\lVert u'\rVert_\infty^2
  -2m\lVert u''\rVert_\infty\geq\sigma.
\]
The required endpoint metrics follow from \(u=0\) near \(0\) and
\(u=a\) near \(L\).
\end{proof}

\section{Das's open Morse--surgery theorem}
\label{sec:das}

Das \cite[Theorem~1.2]{Das2026} proves the following result for
orientable manifolds. We follow the proof in Section~3.1 of his
paper, using Lemma~\ref{lem:scaling-collar} to keep track of the
collar lengths needed for completeness. The construction does not
require orientability.

\begin{theorem}[Open Morse--surgery theorem]\label{thm:open-morse}
Let \(M^N\) be a connected open manifold, with \(N\geq3\). Suppose
that, for some \(a\in\R\), there is a proper Morse function
\[
  f\colon M\longrightarrow[a,\infty)
\]
whose critical points all have index at most \(N-3\). Then \(M\)
admits a complete metric of uniformly positive scalar curvature.
\end{theorem}

\begin{proof}
Fix \(\sigma>0\). Since \(f\) is proper and bounded below, it attains
its minimum \(m\) at finitely many points \(p_1,\ldots,p_s\), all of
index zero. Choose disjoint Morse coordinate neighborhoods \(U_i\)
of these points, in which \(f=m+|x|^2\). For sufficiently small
\(\varepsilon>0\),
\[
  f^{-1}([m,m+\varepsilon])\subseteq\bigcup_{i=1}^sU_i.
\]
Indeed, otherwise properness would give a sequence outside this union
converging to a global minimum outside the union, a contradiction.
Choosing a regular value \(a_0>m\) sufficiently close to \(m\), we
therefore obtain a sublevel set \(K_0=f^{-1}([a,a_0])\) consisting
of finitely many disjoint \(N\)-balls.

Extend \(a_0\) to a sequence of regular values
\(a_0<a_1<a_2<\cdots\) tending to infinity, and write
\[
  K_j=f^{-1}([a,a_j]),\qquad
  Y_j=\partial K_j=f^{-1}(a_j),\qquad
  W_j=f^{-1}([a_j,a_{j+1}]).
\]
The sets \(K_j\) form a compact exhaustion of \(M\). Each shell
\(W_j\) contains finitely many critical points, since its critical
set is closed and discrete. Equip the balls in \(K_0\) with torpedo
metrics, product near the boundary and scaled to have scalar curvature
at least \(2\sigma\). The induced metric \(h_0\) on \(Y_0\) then
also satisfies \(\Sc_{h_0}\geq2\sigma\).

Suppose the metric has been constructed on \(K_j\), is product near
\(Y_j\), and induces a boundary metric \(h_j\) with
\(\Sc_{h_j}\geq2\sigma\). Apply Theorem~\ref{thm:trace} to
\(W_j\), with incoming metric \(h_j\), to obtain a metric
\(\widehat G_j\) of positive scalar curvature. Write
\(\widehat h_{j+1}\) for its outgoing boundary metric. Compactness
gives a constant \(\mu_j>0\) such that
\(\Sc_{\widehat G_j}\geq\mu_j\). Set
\[
  \rho_j=\min\left\{1,\sqrt{\frac{\mu_j}{2\sigma}}\right\},
  \qquad
  h_{j+1}=\rho_j^2\widehat h_{j+1}.
\]
Then \(\rho_j^2\widehat G_j\) has scalar curvature at least
\(2\sigma\), with incoming boundary metric \(\rho_j^2h_j\) and
outgoing boundary metric \(h_{j+1}\). In the rescaled normal
coordinate \(s=\rho_jt\), its collars take the form
\[
  \rho_j^2h_j+ds^2
  \qquad\text{and}\qquad
  h_{j+1}+ds^2.
\]
In particular, \(\Sc_{h_{j+1}}\geq2\sigma\).

Lemma~\ref{lem:scaling-collar}, applied with \(\ell=1\), gives a
collar from \(h_j\) to \(\rho_j^2h_j\), of length \(L_j\geq1\)
and scalar curvature at least \(\sigma\). Attach this collar to
the incoming boundary of the scaled trace. Inserting a boundary
collar does not change the diffeomorphism type of \(W_j\); the
identification can be chosen relative to the boundary. Thus the
metric extends over \(K_{j+1}\) without changing it on \(K_j\).
All gluings are smooth because the metrics are product near the
gluing hypersurfaces. Induction gives a smooth metric \(G\) on
\(M\) with \(\Sc_G\geq\sigma\).

To prove completeness, let \(\gamma\) be a piecewise smooth curve
that eventually leaves every compact subset of \(M\). Choose
\(j_0\) with \(\gamma(0)\in\Int K_{j_0}\). For each \(j\geq j_0\),
the curve must pass from \(K_j\) to the exterior of \(K_{j+1}\)
and hence cross a component of the collar inserted in \(W_j\).
On that collar, \(G=ds^2+f_j(s)^2h_j\), so a crossing from
\(s=0\) to \(s=L_j\) has length at least \(L_j\geq1\).
The collars have disjoint interiors, and therefore
\[
  L_G(\gamma)\geq\sum_{j\geq j_0}1=\infty.
\]
The divergent-curve criterion for completeness, equivalently
Hopf--Rinow, now shows that \(G\) is complete. None of the
construction requires orientability.
\end{proof}

\section{The three-factor theorem}
\label{sec:main-proof}

\begin{proof}[Proof of Theorem~\ref{thm:three-factors}]
For each \(i\), Proposition~\ref{prop:Morse-exhaustion-no-max}
provides a proper Morse exhaustion
\(f_i\colon X_i\to[0,\infty)\) whose critical indices are at most
\(n_i-1\). Set
\[
  N=n_1+n_2+n_3,\qquad F=f_1+f_2+f_3.
\]
By Proposition~\ref{prop:product-index-bound}, \(F\) is a proper
Morse exhaustion and every critical point \(p\) satisfies
\[
  \ind_F(p)\leq(n_1-1)+(n_2-1)+(n_3-1)=N-3.
\]
Theorem~\ref{thm:open-morse} therefore gives a complete metric of
uniformly positive scalar curvature. By rescaling, we get
\(\Sc\geq1\).
\end{proof}

\subsection{The two-factor case}\label{sec:two-factors}

A uniform positive lower bound is already impossible for $\R \times \R$.

\begin{proposition}\label{prop:two-factor-obstruction}
There is no complete Riemannian metric \(g\) on \(\R\times\R\)
with \(\inf\Sc_g>0\).
\end{proposition}

\begin{proof}
If \(\Sc_g\geq c>0\), the Gaussian curvature of \((\R^2,g)\)
satisfies \(K_g=\tfrac12\Sc_g\geq c/2\). Since
\(\Ric_g=K_gg\) in dimension two, Bonnet--Myers would make a
complete such surface compact, a contradiction.
\end{proof}

Dropping the uniform lower bound does not make the two-factor
statement true. The following recent result allows us to construct a counterexample to Problem~\ref{prob:products}.

\begin{theorem}[{\cite[Theorem~1.2]{YZ2026}}]
\label{thm:handlebody-obstruction}
Let \(H_\gamma\) be a compact three-dimensional handlebody of genus \(\gamma\).
If \(\Int H_\gamma\) admits a complete metric of nonnegative
scalar curvature, then \(\gamma\leq1\).
\end{theorem}

\begin{example}\label{ex:two-factor-counterexample}
Let \(D_1,D_2,D_3\) be pairwise disjoint closed disks in \(S^2\),
and set
\[
  X_1=(0,1),\qquad
  X_2=S^2\setminus(D_1\cup D_2\cup D_3),
\]
so that \(X_2\) is the interior of the compact surface
\[
  \overline X_2
  =S^2\setminus(\Int D_1\cup\Int D_2\cup\Int D_3).
\]
Then
\[
  X_1\times X_2\cong\Int\bigl([0,1]\times\overline X_2\bigr).
\]

 To see the topology of this compactification explicitly, choose a
handle decomposition of the pair of pants \(\overline X_2\) consisting
of one \(2\)-dimensional \(0\)-handle and two \(2\)-dimensional
\(1\)-handles. Taking the product with \([0,1]\) thickens the
\(0\)-handle to a \(3\)-ball and each \(1\)-handle to a
\(3\)-dimensional \(1\)-handle. Consequently, after rounding the
corners, \([0,1]\times\overline X_2\) is obtained from a \(3\)-ball by
attaching two \(1\)-handles, and hence is a compact handlebody \(H_2\)
of genus two.

 Thus
\(X_1\times X_2\cong\Int H_2\), and
Theorem~\ref{thm:handlebody-obstruction} rules out a complete metric
of nonnegative scalar curvature on this product.
\end{example}

\bibliographystyle{amsalpha}
\bibliography{references}

\end{document}